\documentclass[11pt]{article}

\usepackage{amsmath,amssymb,amsthm}
\usepackage[dvips]{graphicx}
\usepackage{latexsym}
\usepackage{eucal}
\usepackage[pdftex]{hyperref}
\usepackage[latin1]{inputenc}
\usepackage[english,activeacute]{babel}

\newtheorem{theorem}{Theorem}[section]
\newtheorem{proposition}[theorem]{Proposition}
\newtheorem{lemma}[theorem]{Lemma}
\newtheorem{corollary}[theorem]{Corollary}	

\theoremstyle{definition}

\newtheorem{remark}[theorem]{Remark}

\newtheorem{example}[theorem]{Example}

\title{{\bf\Large Existence of solutions for some nonlinear problems with boundary value conditions}}
\author{{\large Dionicio Pastor Dallos Santos }\footnote{Email: dionicio@ime.usp.br}\hspace{2mm}
{\bf\large}\vspace{1mm}\\
{\small Department of Mathematics, IME-USP, Cidade Universit\'aria,}\\
{\small  CEP 05508-090, S\~ao Paulo, SP, Brazil}
}

\date{}
\begin{document}
\maketitle

\begin{abstract}
In this paper we study the existence of solutions for  nonlinear boundary value problems
\[
\left\{\begin{array}{lll}
(\varphi(u' ))'   = f(t,u,u') & & \\
l(u,u')=0, & & \quad \quad 
\end{array}\right.
\] 
where $l(u,u') =0$ denotes  the Dirichlet or mixed  conditions on $\left[0, T\right]$, $\varphi$  is a   bounded, singular or classic
homeomorphism such that $\varphi(0)=0$, $f:\left[0, T\right]\times \mathbb{R} \times \mathbb{R}\rightarrow \mathbb{R} $ is a continuous function, and $T$ a positive real number. All the  contemplated  boundary value problems are reduced  to finding a fixed point for one  operator  defined on a  space of  functions, and  Schauder fixed point theorem or Leray-Schauder degree are used.
\end{abstract}

 \medskip

\noindent 
Mathematics Subject Classification (2010). 34B15; 34B16; 47H11. 

\noindent 
Key words:  Dirichlet problem, Schauder fixed point theorem, Leray-Schauder degree, mixed boundary value problems.


\section{Introduction}
The purpose of this article is to obtain  some existence results for nonlinear boundary value problems of the form
\begin{equation}\label{equa1}
\left\{\begin{array}{lll}
(\varphi(u' ))'   = f(t,u,u') & & \\
l(u,u')=0, 
\end{array}\right.
\end{equation}
where $l(u,u') =0$ denotes the Dirichlet or mixed  boundary conditions  on the interval  $\left[0, T\right]$,
 $\varphi$  is a  bounded, singular  or classic   homeomorphism such that $\varphi(0)=0$,  $f:\left[0, T\right]\times \mathbb{R} \times \mathbb{R}\rightarrow \mathbb{R}$ is a continuous function, and $T$ a positive real number.
 
Recently, the problem (\ref{equa1}) in special cases, when $\varphi$ is an  increasing  homeomorphism  from $(-a, a)$ to $\mathbb{R}$ such that $\varphi(0)=0$ and $l(u,u')=0$ denotes the periodic, Neumann or Dirichlet boundary conditions, has been investigated by  C. Bereanu and  J. Mawhin in \cite{ma1}.

 In \cite{man8}, the authors have studied the problem (\ref{equa1}), where $\varphi:\mathbb{R} \rightarrow (-a, a) \ (0<a\leq \infty)$ and  periodic  boundary conditions. They obtained the existence of solutions by means of the Leray-Schauder degree theory.


The  paper is organized as follows. In Section 2, we introduce some notations and  preliminaries, which   will be crucial in the proofs of our results. Section 3 is devoted to the study of existence of solutions for the  Dirichlet problems with  bounded homomorphisms
\[
\left\{\begin{array}{lll}
(\varphi(u' ))'   = f(t,u,u') & & \\
u(0)=0=u(T). & & \quad \quad 
\end{array}\right.
\] 
In particular, C. Bereanu and  J. Mawhin  in  \cite{ma} proved the existence  of  at least one solution  by  means of  the  Leray-Schauder degree.
\begin{theorem}[Bereanu and Mawhin]\label{teo1} If the function  $f$ satisfies the condition                       
\begin{center}
$ \exists \, c>0 \text{ such that }  \left|f(t,x,y)\right|\leq c < \frac{a}{2T},$ \quad $\forall (t,x,y)\in[0, T]\times \mathbb{R}\times\mathbb{R}$,
\end{center}
the Dirichlet problem  has  at least  one solution.
\end{theorem}
The main purpose of this section is an extension of the results obtained  in the previous theorem.  For this, we use  topological methods based upon Leray-Schauder degree \cite{man12} and more general properties of the function $f$. In Section 4, we  use  the fixed point theorem of Schauder to show the existence of at least one solution  for boundary value problems of the type
\[
\left\{\begin{array}{lll}
(\varphi(u' ))'   = f(t,u,u') & & \\
u(T)=u(0)=u'(T). & & \quad \quad 
\end{array}\right.
\]
where $\varphi:(-a,a)\rightarrow \mathbb{R}$ (we call it singular). We call \textsl{solution } of this problem any function $u:\left[0, T\right]\rightarrow \mathbb{R}$ of class $C^{1}$ such that $\displaystyle \max_{\left[0, T\right]}\left|u'(t)\right|<a$, satisfying  the boundary conditions and  the function \ $\varphi(u')$ is continuously differentiable and $(\varphi(u'(t) ))'  = f(t,u(t),u'(t))$ for all $t\in\left[0,T\right]$. In Section 5, for  $u(T)=u'(0)=u'(T)$ boundary  conditions  and  classic homeomorphisms $(\varphi:\mathbb{R} \rightarrow \mathbb{R})$, we investigate the existence of at least one solution 
using  Leray-Schauder degree, where a  \textsl{solution } of this problem is any function  $u:\left[0, T\right]\rightarrow \mathbb{R}$ of 
class $C^{1}$ such that $\varphi(u')$ is continuously differentiable, which satisfies the boundary conditions and $(\varphi(u'(t) ))'  = f(t,u(t),u'(t))$ for all $t\in\left[0,T\right]$. Such problems do not seem to have been studied in the literature. In the present paper generally we follow the ideas of Bereanu and Mawhin \cite{ man5, man4, man8, ma1, ma, man9}.


\section{Notation and preliminaries}
\label{S:-1}
We first introduce some notation. For fixed $T$, we denote  the usual norm in $L^{1}=L^{1}(\left[0,T\right], \mathbb R)$ for $\left\| \cdot \right\|_{L^{1}}$. For $C=C(\left[0,T\right], \mathbb R)$ we indicate the Banach space of all  continuous functions from $\left[0, T\right]$ into $\mathbb R$ witch the norm  $\left\| \cdot \right\|_{\infty}$, $C^{1}=C^{1}(\left[0,T\right], \mathbb R)$  denote the Banach space of continuously differentiable functions  from $\left[0, T\right]$ into $\mathbb R$ endowed witch the usual norm $\left\|u\right\|_{1}=\left\|u\right\|_{\infty} +  \left\|u' \right\|_{\infty}$ and for   $C^{1}_{0}$  we  designate the closed subspace of $C^{1}$ defined by  $C^{1}_{0}=\left\{u\in C^{1}:u (T)=0=u(0)\right\}$. 

We introduce the following applications:
\medskip

\noindent 
the  \textit{Nemytskii operator} $N_f:C^{1} \rightarrow C $, 
\begin{center}
$N_f (u)(t)=f(t,u(t),u'(t))$, 
\end{center}
the  \textit{integration operator}   $H:C \rightarrow C^{1}$, 
\begin{center}
$ H(u)(t)=\int_0^t u(s)ds$, 
\end{center}
\noindent 
the following continuous linear applications:
\begin{center}
$K:C \rightarrow C^{1}, \ \  K(u)(t)=-\int_t^T u(s)ds $,
\end{center}
\begin{center}
$Q:C \rightarrow C, \ \  Q(u)(t)=\frac{1}{T}\int_0^T u(s)ds$,
\end{center}
\begin{center}
$S:C \rightarrow C, \ \  S(u)(t)=u(T)$,
\end{center}
\begin{center}
$P:C \rightarrow C, \ \  P(u)(t)=u(0)$.
\end{center}

For  $u\in C$, we write
\begin{center}
$ u_{m}=\displaystyle  \min_{[0,T]} u,  \  u_{M}= \displaystyle \max_{[0,T]}u,  \   u^{+}=\displaystyle \max\left\{ u,0 \right\},  \  u^{-}=\displaystyle \max \left\{ -u,0 \right\}$.
\end{center}
 
The following lemma is an adaptation of a result of \cite{ma1} to the case of a homeomorphism which is not defined everywhere. We present here the demonstration for better understanding of the development of our research.
\begin{lemma}\label{lebema}
Let $B=\left\{h\in C:\left\|h \right\|_{\infty}<a/2\right\}$. For  each $h\in B$,  there exists a unique  $Q_{\varphi}=Q_{\varphi}(h)\in Im(h)$ (where $Im(h)$ denotes the range of  $h$) such that
\begin{center}
$\int_0^T \varphi^{-1}(h(t)-Q_{\varphi}(h))dt = 0$.
\end{center}
Moreover, the function    $Q_{\varphi}:B\rightarrow \mathbb{R}$ is continuous  and sends bounded sets into bounded sets.
\end{lemma}
\begin{proof} 
Let $h\in B$. We define the continuous application $G_{h}: \left[h_{m}, h_{M}\right] \rightarrow \mathbb{R}$  for
\begin{center}
$G_{h}(s)=\int_0^T \varphi^{-1}(h(t)-s)dt$.
\end{center}
We now show that the equation
\begin{equation}\label{paz}
G_{h}(s)=0
\end{equation}
has a unique solution  $Q_{\varphi}(h)$. Let  $r, \  s \in \left[h_{m}, h_{M}\right] $ be such that
\begin{center}
$\int_0^T \varphi^{-1}(h(t)-r)dt=0, \  \int_0^T \varphi^{-1}(h(t)-s)dt=0$,
\end{center}
that is
\begin{center}
$\int_0^T \varphi^{-1}(h(t)-r)dt=\int_0^T \varphi^{-1}(h(t)-s)dt$.
\end{center}
It follows that there exist  $\tau \in \left[0, T\right]$  such that
\begin{center}
$\varphi^{-1}(h(\tau)-r)=\varphi^{-1}(h(\tau)-s)$.
\end{center}
Using the injectivity  of  $\varphi^{-1}$  we deduce that  $r=s$. Let us now show the existence. Because   $\varphi^{-1}$  is strictly monotone  and  $\varphi^{-1}(0)=0$, we have that
\begin{center}
$G_{h}(h_{m})G_{h}(h_{M})\leq0$.
\end{center}
It follows that there exists  $s\in\left[h_{m}, h_{M}\right]$ such that  $G_{h}(s)=0$. Consequently for each  $h\in B$, the equation (\ref{paz}) has a unique solution. Thus, we define the function $Q_{\varphi}: B \rightarrow \mathbb{R} $ such that
\begin{center}
$\int_0^T \varphi^{-1}(h(t)-Q_{\varphi}(h))dt = 0$.
\end{center}

On the other hand, because $h\in B$, we have that  
\begin{center}
$\left|Q_{\varphi}(h)\right|\leq \left\|h\right\|_{\infty}<a/2$.
\end{center}
Therefore, the function  $Q_{\varphi}$  sends bounded sets into bounded sets.

Finally, we show that  $Q_{\varphi}$ is continuous on $B$. Let $\left(h_{n}\right)_{n}\subset C$  be a sequence such 
that $h_{n}\rightarrow h$ in $C$. Since the function  $Q_{\varphi}$  sends bounded sets into bounded sets, then  $\left(Q_{\varphi}(h_{n})\right)_{n}$ is bounded. Hence, $\left(Q_{\varphi}(h_{n})\right)_{n}$ is relatively compact. Without loss of generality, passing   if necessary to a subsequence, we can  assume that
\begin{center}
$ \displaystyle \lim_{n \to \infty} Q_{\varphi}(h_{n})=\widetilde{a} $,
\end{center}
where for each $n\in \mathbb{N}$  we obtain 
\begin{center}
$ \displaystyle \lim_{n \to \infty}\int_0^T \varphi^{-1}(h_{n}(t)-Q_{\varphi}(h_{n}))dt=0 $.
\end{center}
Using the  dominated convergence theorem, we deduce that
\begin{center}
$ \int_0^T \varphi^{-1}(h(t)-\widetilde{a})dt=0$,
\end{center}
so we have that $Q_{\varphi}(h)=\widetilde{a}$.  Hence, the function  $Q_{\varphi}$ is continuous.
\end{proof}
The following extended homotopy invariance property of the Leray-Schauder degree, can be found in   \cite{man6}.
\begin{proposition}\label{propo1}
Let $X$  be a real Banach space, $V\subset [0,1]\times X$ be an open, bounded set and  $M$ be a completely continuous operator  on  $\overline{V}$  such that $x\neq M(\lambda,x)$  for each  $(\lambda,x)\in\partial V$. Then the Leray-Shauder degree
\begin{center}
$ deg_{LS}(I-M(\lambda,.),V_{\lambda},0)$
\end{center}
is well defined  and independent of   $\lambda$  in   $[0,1]$, where  $V_{\lambda}$ is the open, bounded (possibly empty) set defined  by  $V_{\lambda}=\left\{x\in X:(\lambda,x)\in V\right\}$.
\end{proposition} 


\section{Dirichlet problems with  bounded homeomorphisms}
\label{S:0}

In this section we are interested  in  Dirichlet boundary value problems of the type
\begin{equation}\label{diri1}
\left\{\begin{array}{lll}
(\varphi(u' ))'   = f(t,u,u') & & \\
u(0)=0=u(T), 
\end{array}\right.
\end{equation}
where  $\varphi:\mathbb{R} \rightarrow (-a,a)$ is  a homeomorphism such that $\varphi(0)=0$ and $f:\left[0, T\right]\times \mathbb{R} \times \mathbb{R}\rightarrow \mathbb{R} $ is a continuous function. In order to apply Leray-Schauder degree theory  to show  the existence of  at least one  solution of (\ref{diri1}), we introduce, for  $ \lambda \in [0,1]$, the family of Dirichlet boundary value problems
\begin{equation}\label{diri2}
\left\{\begin{array}{lll}
(\varphi(u' ))'   = \lambda f(t,u,u') & & \\
u(0)=0=u(T). 
\end{array}\right.
\end{equation}
Let
\begin{center}
$\Omega = \left\{(\lambda,u)\in[0,1]\times C^{1}_{0}:\left\|\lambda H(N_{f}(u))\right\|_{\infty}<a/2\right\}$. 
\end{center}
Clearly $\Omega$ is  an open set in $[0,1]\times C^{1}_{0}$, and  is  nonempty because $\left\{0\right\}\times C^{1}_{0}\subset\Omega$. Using Lemma \ref{lebema}, we can define the operator  $M:\Omega \rightarrow C^{1}_{0}$ by 
\begin{equation}\label{diri3}
 M(\lambda,u)=H(\varphi^{-1}\left[ \lambda H(N_f(u))- Q_{\varphi}(\lambda H(N_f(u)))\right]). 
\end{equation}    
Here  $\varphi^{-1}$  with an abuse of notation is understood as the operator $\varphi^{-1}:B_{a}(0) \subset C \rightarrow C$ defined by $\varphi^{-1}(v)(t)=\varphi^{-1}(v(t))$. It is clear that   $\varphi^{-1}$ is  continuous and sends bounded sets into bounded sets.

When the boundary conditions are periodic or Neumann, an operator has been considered by Bereanu and Mawhin \cite{ma}.

The following lemma plays a pivotal role to study the solutions of the  problem (\ref{diri2}).
\begin{lemma}\label{dallos1}
The operator  $M$ is well  defined  and continuous. Moreover, if  $(\lambda,u)\in \Omega$  is such that  $M(\lambda,u)=u$, then   $u$ is solution of  (\ref{diri2}).
\end{lemma}
\begin{proof}
Let $(\lambda,u)\in \Omega$. We show that in fact $M(\lambda,u)\in C^{1}_{0}$. It is clear that
\begin{center}
 $\left(M(\lambda,u)\right)^{'}=\varphi^{-1}\left[\lambda H(N_f(u))- Q_{\varphi}(\lambda H(N_f(u))) \right]$,
\end{center}
where the  continuity of $M(\lambda, u) $ and $\left(M(\lambda,u)\right)^{'}$ follows from  the continuity of the applications $H$ and  $N_{f}$. 
 
On the other hand using  Lemma \ref{lebema}, we have
\begin{center}
 $  M(\lambda,u)(0)=0= M(\lambda,u)(T)$.
\end{center}
Therefore   $M(\Omega)\subset  C^{1}_{0}$ and  $M$ is well defined. The continuity of $M$  follows by the continuity of the operators which compose it $M$.

 Now suppose that   $(\lambda,u)\in \Omega$  is such that  $M(\lambda,u)=u$. It follows from (\ref{diri3}) that
\begin{center}
 $ u(t)= M(\lambda,u)(t)=H(\varphi^{-1}\left[ \lambda H(N_f(u))- Q_{\varphi}(\lambda H(N_f(u)))\right])(t) $
\end{center}
 for all  $t\in\left[0,T\right]$. Differentiating, we obtain
\begin{center}
 $  u'(t)= \varphi^{-1}\left[\lambda H(N_f(u))- Q_{\varphi}(\lambda H(N_f(u)))\right](t)$.
\end{center}
Applying   $\varphi$  to both of its members and differentiating again, we deduce that
\begin{center}
 $(\varphi(u'(t) ))'=\lambda N_{f}(u)(t)$
\end{center}
for all $t\in\left[0, T\right]$. Thus, $u$ satisfies problem (\ref{diri2}). This completes the proof.
\end{proof}
\begin{remark} 
Note that the reciprocal of Lemma \ref{dallos1} is not true because we can not guarantee that $\left\|\lambda H(N_{f}(u)\right\|_{\infty} <a/2$  for  $u$ solution  of (\ref{diri2})
\end{remark}
In our main result, we need the following lemma to obtain the required  a priori bounds for the possibles fixed points of $M$.
\begin{lemma}\label{dallos}
Assume that there exist  $h\in C([0,T],\mathbb{R^{+}})$ \ and \  $n\in C^{1}(\mathbb{R},\mathbb{R})$ such that 
\begin{center}
$\left\|h \right\|_{L^{1}}<a/2,  \   \varphi(y)n'(x)y\geq 0 ,  \   n(0)=0$,
\end{center}
and 
\begin{center}
$\left|f(t,x,y)\right|\leq f(t,x,y)n(x) + h(t)$
\end{center}
for all  $(t,x,y)\in[0,T]\times \mathbb{R} \times \mathbb{R}  $. If  $(\lambda,u)\in \Omega$ is such that $M(\lambda,u)=u$, then                
\begin{center}
$\left\|\lambda H(N_{f}(u))\right\|_{\infty}\leq\left\|h \right\|_{L^{1}},  \   \left\|u' \right\|_{\infty}\leq L$ \   and  \   $\left\| u \right\|_{1}\leq L +LT$,
\end{center}
where  $L= \max \left\{\left|\varphi^{-1}(-2\left\|h \right\|_{L^{1}})\right|,\left|\varphi^{-1}(2\left\|h \right\|_{L^{1}})\right|\right\}$.
\end{lemma}
\begin{proof} 
Let  $\lambda\neq 0$ and  $(\lambda,u)\in \Omega$ be such that  $M(\lambda,u)=u$. Using Lemma \ref{dallos1}, we have  that  $u$ is solution of  (\ref{diri2}), which implies that 
\begin{center} 
 $\varphi(u')= \lambda H(N_f(u))-Q_{\varphi}(\lambda H(N_f(u)))$, \   $u(0)=0=u(T)$,
 \end{center}
where for all  $t\in[0,T]$, we obtain
\begin{align*}
\left|\lambda H(N_{f}(u))(t)\right|  &\leq  \int_0^T \left|f(s,u(s),u'(s))\right|ds  \\
&\leq \int_0^T f(s,u(s),u'(s))n(u(s)) ds +\int_0^T h(s)ds.
\end{align*}

On the other hand, because  $\varphi$  is a homeomorphism such that
\begin{center}  
$\varphi(y)n'(x)y\geq 0$
\end{center}
for all  $x, y \in \mathbb{R}$, then 
\begin{center}
$-\int_0^T \varphi(u'(t))n'(u(t))u'(t) dt \leq 0$.
\end{center}
Using the integration by parts formula, the boundary conditions and the fact that  $n(0)=0$,  we deduce that 
\begin{center}
$\int_0^T (\varphi(u'(t)))'n(u(t)) dt = - \int_0^T \varphi(u'(t))n'(u(t))u'(t) dt \leq 0$.
\end{center}
Since  $\lambda\in (0, 1]$ and  $u$ is solution of  (\ref{diri2}), we  have that
\begin{center}
$\int_0^T f(t,u(t),u'(t))n(u(t))dt \leq0$,
\end{center}
and hence
\begin{center}
$\left|\lambda H(N_{f}(u))(t)\right|\leq\left\|h \right\|_{L^{1}}$.
\end{center}

On the other hand, since  that  $Q_{\varphi}(\lambda H(N_f(u))) \in Im(\lambda H(N_f(u)))$, we get
\begin{center}
$ \left|\varphi(u'(t))\right|\leq 2\left\|h \right\|_{L^{1}}$  
\end{center}
for all $t\in [0,T]$. It follows that   
\begin{center}
$\left\|\varphi(u')\right\|_{\infty}\leq 2\left\|h \right\|_{L^{1}}$,
\end{center}
 which implies that  $\left\|u' \right\|_{\infty}\leq L$, where  $L= $max$\left\{\left|\varphi^{-1}(-2\left\|h \right\|_{L^{1}})\right|,\left|\varphi^{-1}(2\left\|h \right\|_{L^{1}})\right|\right\}$.
Using again the boundary conditions, we have that
\begin{center}
$\left|u(t)\right|\leq \int_0^t\left|u'(s)\right|ds \leq \int_0^T\left|u'(s)\right|ds\leq LT  \quad (t\in\left[0,T\right])$,
\end{center}
and hence 
\begin{center}
$\left\| u \right\|_{1}\leq L +LT$.     
\end{center}
Finally, if  $u=M(0, u)$, then   $u=0$, so the proof is complete.
\end{proof}
Let $\rho, \  \kappa \in \mathbb{R}$ be such that  $\left\|h\right\|_{L^{1}} <\kappa<a/2,  \  \rho>L +LT$ and consider the set 
\begin{center}
$V=\left\{(\lambda,u)\in[0,1]\times C^{1}_{0}:\left\|\lambda H(N_{f}(u))\right\|_{\infty}<\kappa, \  \left\|u\right\|_{1}<\rho\right\}$.
\end{center}
Since  the set  $\left\{0\right\}\times\left\{u\in C^{1}_{0}:\left\|u\right\|_{1}<\rho\right\}\subset V$, then we deduce that $V$ is nonempty. Moreover, it is clear that $V$ is open and bounded in  $ [0,1]\times C^{1}_{0}$  and  $\overline{V}\subset\Omega$. On the other hand using an argument similar to the one introduced in the proof of Lemma 6 in \cite{ma}, it is not difficult to see that $M:\overline{V}\rightarrow C^{1}_{0}$  is well defined, completely continuous and
\begin{center}
$ u\neq M(\lambda,u)$  \  for all  \ $(\lambda,u) \in\partial V$.
\end{center}
\subsection{Existence results}
\label{work2}
In this subsection, we present and prove our main result.
\begin{theorem}\label{teopri}
If  $f$  satisfies conditions of Lemma \ref{dallos}, then  problem (\ref{diri1}) has at least one solution.
\end{theorem}
\begin{proof} 
   Let $M$ be the operator given by (\ref{diri3}). Using Proposition \ref{propo1}, we deduce that
\begin{center}
$deg_{LS}(I-M(0,.),V_{0},0)=deg_{LS}(I-M(1,.),V_{1},0)$, 
\end{center}
where  $deg_{LS}(I-M(0,.),V_{0},0)=deg_{LS}(I,B_{\rho}(0),0)=1$. Thus, there exists $u \in V_{1}$ such that $ M(1,u)=u$, which is a solution for (\ref{diri1}). 
\end{proof}                                                                                         
\begin {remark}
 Note that Theorem   \ref{teopri}  is a generalization of Theorem \ref{teo1}.
\end {remark}
\begin{corollary}\label{exemp} 
Assume that  $\varphi$ is  an increasing homomorphism. Let  $h\in C([0,T],\mathbb{R^{+}})$ be such that 
\begin{center}
$\left\|h \right\|_{L^{1}}<a/2,  \  \left|f(t,x,y)\right| \leq f(t,x,y)x + h(t)$
\end{center}
for all  $x,y\in \mathbb{R} $ and $t\in [0,T]$. If $(\lambda,u)\in \Omega$ is such that $M(\lambda,u)=u$, then 
\begin{center}
$\left\|\lambda H(N_{f}(u))\right\|_{\infty}\leq\left\|h \right\|_{L^{1}}, \  \left\|u' \right\|_{\infty}\leq L$ \ e \ $\left\| u \right\|_{1}\leq L +LT$,
\end{center} 
where $L= \max \left\{\left|\varphi^{-1}(-2\left\|h \right\|_{L^{1}})\right|,\left|\varphi^{-1}(2\left\|h \right\|_{L^{1}})\right|\right\}$.
\end{corollary} 
\begin{proof} Since   $\varphi$  is an increasing homomorphism we have that               
\begin{center}
$\varphi(y)y\geq 0$
\end{center}
for all $y\in \mathbb{R}$. Using  Lemma \ref{dallos}  with  $n(x)=x$  \ for all \  $x\in \mathbb R$, we can obtain the conclusion of Corollary \ref{exemp}. The proof is achieved. 
\end{proof}
\begin{theorem}\label{gallo}
If  $f$  satisfies conditions of Corollary \ref{exemp}, then  problem (\ref{diri1}) has at least one solution.
\end{theorem}

Let us give now an application of Theorem \ref{gallo}  when $f$ is  unbounded.
\begin{example}
Let $\varphi:\mathbb{R} \rightarrow (-1,1)$ \ with \ $\varphi(y)=\frac{y}{\sqrt{1+y^2}}$, $f(t,x,y)=x-2$ \  and \ $h(t)=4$. Using  Theorem \ref{gallo}, we obtain that the problem
\begin{center}
$\left(\frac{u'}{\sqrt{1+u'^2}}\right)'=u-2, \  u(0)=u(T)=0$,
\end{center}
has at least one solution if $T<1/8$.
\end{example}


\section{Problems with singular homeomorphisms and tree-point boundary conditions}
\label{S:2}
In this section we  study the existence of at least one solution for boundary value problems of  the  type
\begin{equation}\label{misto1}
\left\{\begin{array}{lll}
(\varphi(u' ))'   = f(t,u,u') & & \\
u(T)=u(0)=u'(T), 
\end{array}\right.
\end{equation}
where  $\varphi: (-a,a) \rightarrow  \mathbb{R}$ is  a homeomorphism such that $\varphi(0)=0$ and $f:\left[0, T\right]\times \mathbb{R} \times \mathbb{R}\rightarrow \mathbb{R} $ is a continuous function.

In order to transform problem (\ref{misto1}) to a fixed point problem we  use a similar argument introduced in Lemma \ref{lebema}  for $h\in C$.
\begin{lemma}\label{dallos11}
$u \in C^{1}$ is a solution of  (\ref{misto1}) if and only if  $u$ is a fixed point of the operator   $M$ defined  on $C^{1}$ by
\begin{center}
$M(u)=\varphi^{-1}(-Q_{\varphi}(K(N_{f}(u)))) + H\left(\varphi^{-1} \left[K(N_{f}(u))-Q_{\varphi}(K(N_{f}(u)))\right]\right)$.
\end{center}
\end {lemma}
\begin{proof} 
 If  $u\in C^{1}$ is solution of (\ref{misto1}), then   
\begin{center}
$(\varphi(u'(t)))' = N_f (u)(t)=f(t,u(t),u'(t))$, \  $u(0)=u(T)$,\ $u(0)=u'(T)$
\end{center}
 for all $t\in\left[0, T\right]$. Applying  $K$ to both  members  and using the fact that  $u(0)=u'(T)$, we deduce that  
\begin{center}
 $\varphi(u'(t)) = \varphi(u(0)) + K(N_{f}(u))(t)$.
\end{center}
Composing with the function  $\varphi^{-1}$, we obtain   
\begin{center}
$u'(t) =\varphi^{-1}\left[ K(N_{f}(u))(t)+c\right]$,   
\end{center}
where $ c=\varphi (u(0))$. Integrating from 0 to  $t\in\left[0, T\right]$, we  have that               
\begin{center}
 $ u(t)=u(0)+ H\left( \varphi^{-1} \left[K(N_{f}(u))+c \right]\right)(t)$. 
\end{center}
Because  $u(0)=u(T)$, then                 
\begin{center}
$\int_0^T \varphi^{-1}\left[K(N_{f}(u))(t)+c\right]dt=0$.
\end{center} 
Using an argument similar  to  the   introduced in Lemma  \ref{lebema}, it follows that $c=-Q_{\varphi}(K(N_{f}(u)))$. Hence,                      
\begin{center}
$u=\varphi^{-1}(-Q_{\varphi}(K(N_{f}(u))))+ H \left( \varphi^{-1}\left[ K(N_{f}(u))-Q_{\varphi}(K(N_{f}(u)))\right] \right) $.
\end{center}

Let $u\in C^{1}$ be such that $ u=M(u)$. Then
\begin{equation}\label{ptof}
u(t)=\varphi^{-1}(-Q_{\varphi}(K(N_{f}(u))))+ H\left( \varphi^{-1}\left[K(N_{f}(u))-Q_{\varphi}(K(N_{f}(u)))\right]\right)(t)
\end{equation} for all $t\in\left[0, T\right]$. Since  $\int_0^T \varphi^{-1}\left[K(N_{f}(u))(t)-Q_{\varphi}(K(N_{f}(u)))\right]dt=0$, therefore, we have that  $u(0)=u(T)$. Differentiating  (\ref{ptof}), we obtain that 
\begin{center}
$u'(t)=\varphi^{-1}\left[K(N_{f}(u))-Q_{\varphi}(K(N_{f}(u)))\right](t)$. 
\end{center}
 In particular,
\begin{center}
$u'(T)=\varphi^{-1}(0-Q_{\varphi}(K(N_{f}(u))))= \varphi^{-1}(-Q_{\varphi}(K(N_{f}(u))))=u(0)$.
\end{center}
Applying $\varphi$ to both members and  differentiating again, we deduce that
\begin{center}
$(\varphi(u'(t)))' = N_f (u)(t)$, \ \  $u(0)=u(T)$, \ $u(0)=u'(T)$
\end{center}
 for all $t\in\left[0, T\right]$. This completes the proof.
\end{proof} 
\begin{lemma}\label{parecido}
The operator   $M:C^{1} \rightarrow C^{1}$ is completely continuous.
\end{lemma}
\begin{proof}
Let $\Lambda \subset C^{1}$  be a bounded set. Then, if  $u\in\Lambda$, there exists a constant  $\rho>0$ such that 
\begin{equation}\label{dp}
\left\| u\right\|_{1}\leq \rho.
\end{equation}
Next, we show that $\overline{M(\Lambda)}\subset  C^{1}$ is a compact set. Let $(v_{n})_{n} $  be a sequence in   $M(\Lambda)$, and let 
 $(u_{n})_{n}$  be a sequence  in  $\Lambda$ such that  $v_{n}=M(u_{n})$. Using (\ref{dp}), we have that there exists a constant  $L>0$ such that, for all $n\in \mathbb{N}$,  
\begin{center}
 $\left\| N_{f}(u_{n})\right\|_{\infty}\leq L$,
\end{center}
which implies that
\begin{center}
$\left\| K(N_{f}(u_{n}))-Q_{\varphi}( K(N_f(u_{n}))) \right\|_{\infty}\leq 2LT$.
\end{center}
Hence the sequence  $( K(N_{f}(u_{n}))-Q_{\varphi}( K(N_f(u_{n}))))_{n}$ is bounded in  $C$. Moreover, for  $t, t_{1}\in\left[0, T\right]$ and for all $n\in \mathbb{N}$, we have that
\begin{align*}
&\left| K(N_{f}(u_{n}))(t)-Q_{\varphi}( K(N_f(u_{n}))) - K(N_{f}(u_{n}))(t_{1}) + Q_{\varphi}( K(N_f(u_{n})))        \right|\\
&\leq \left|-\int_{t}^{T} f(s,u_{n}(s),u_{n}'(s))ds + \int_{t_{1}}^{T} f(s,u_{n}(s),u_{n}'(s))ds\right|\\
&\leq \left|\int_{t_1}^t f(s,u_n(s),u'_n(s))ds\right|\\
&\leq L\left|t-t_1\right|, 
\end{align*} 
which implies that   $( K(N_{f}(u_{n}))-Q_{\varphi}( K(N_f(u_{n}))))_{n}$ is equicontinuous. Thus, by the Arzel\`a-Ascoli  theorem  there is a subsequence of 
$(K(N_{f}(u_{n}))-Q_{\varphi}( K(N_f(u_{n}))))_{n}$, which we call  $( K(N_{f}(u_{n_{j}})) - Q_{\varphi}( K(N_f(u_{n_{j}}))))_{j}$, which is  convergent in $C$. Using  that  $\varphi^{-1}: C \rightarrow B_{a}(0)\subset C$ is continuous it follows from   
\begin{center}
$M(u_{n_{j}})'=\varphi^{-1} \left[K(N_{f}(u_{n_{j}})) - Q_{\varphi}( K(N_f(u_{n_{j}})))\right] $
\end{center}
 that the sequence  $(M(u_{n_{j}})')_{j}$  is  convergent in $C$. Then, passing to a subsequence if necessary, we obtain 
 that   $(v_{n_{j}})_{j}=( M(u_{n_{j}}))_{j}$ is  convergent in  $C^{1}$. Finally, let  $(v_{n})_{n}$  be a sequence in  $\overline{M(\Lambda)}$. Let  $(z_{n})_{n}\subseteq M(\Lambda)$  be such that
\[
\lim_{n \to \infty}\left\| z_{n}-v_{n}\right\|_{1}=0.
\]
Let $(z_{n_{j}})_{j}$ be a subsequence of  $(z_{n})_{n}$  such that  converge to $z$. It follows that  $z\in \overline{M(\Lambda)}$ and  $(v_{n_{j}})_{j}$ converge 
to $z$. This concludes the proof.
\end{proof}
The next result is based on Schauder's fixed point theorem.
\begin{theorem}
Let $f:[0,T]\times \mathbb{R}\times \mathbb{R} \longrightarrow \mathbb{R}$ be continuous. Then  (\ref{misto1})  has at least one solution.
\end {theorem} 
\begin{proof} 
Let $u\in C^{1}$. Then                
\begin{center}
$M(u)=\varphi^{-1}(-Q_{\varphi}(K(N_{f}(u))))+H \left(\varphi^{-1} \left[ K(N_{f}(u))-Q_{\varphi}(K(N_{f}(u)))\right]\right)$,
\end{center}
where
\begin{center}
$M(u)(0)=\varphi^{-1}(-Q_{\varphi}(K(N_{f}(u))))=M(u)(T)$,
\end{center}
\begin{center}
$M(u)'(T)=\varphi^{-1}(-Q_{\varphi}(K(N_{f}(u))))=M(u)(0)$.
\end{center}
Moreover,
\begin{center}
$\left\| M(u)'\right\|_{\infty}=\left\|\varphi^{-1} \left[ K(N_{f}(u))-Q_{\varphi}(K(N_{f}(u)))\right] \right\|_{\infty}<a$
\end{center}
and
\begin{center}
$\left\| M(u)\right\|_{\infty}<a +aT$.
\end{center}
Hence,
\begin{center}
$\left\| M(u)\right\|_{1}=\left\| M(u)\right\|_{\infty}+\left\| M(u)'\right\|_{\infty}<a+aT+a=2a+aT$.
\end{center}
 Because the operator  $M$ is completely continuous and bounded, we can use Schauder's fixed point theorem to deduce the existence  of  at least one fixed point. This, in turn, implies that problem  (\ref{misto1})  has at least one solution. The proof is complete.
\end{proof}


\section{Problems with classic homeomorphisms and tree-point boundary conditions}
\label{S:3}
We finally consider boundary value problems of the form
\begin{equation}\label{misto2}
\left\{\begin{array}{lll}
(\varphi(u' ))'   = f(t,u,u') & & \\
u(T)=u'(0)=u'(T), 
\end{array}\right.
\end{equation} 
where  $\varphi: \mathbb{R} \rightarrow  \mathbb{R}$ is  a homeomorphism such that $\varphi(0)=0$ and $f:\left[0, T\right]\times \mathbb{R} \times \mathbb{R}\rightarrow \mathbb{R} $ is a continuous function. We remember that an  \textsl{solution } of this problem  is any function $u:\left[0, T\right]\rightarrow \mathbb{R}$ of class $C^{1}$ such that  $\varphi(u')$ is continuously differentiable, satisfying  the   boundary  conditions   and $(\varphi(u'(t) ))'  = f(t,u(t),u'(t))$ for all $t\in\left[0,T\right]$.

Let us consider the operator
\begin{center}
$M_{1}:C^{1}\rightarrow C^{1}$,

$u \mapsto S(u)+ Q(N_{f}(u)) + K(\varphi^{-1}\left[H( N_f (u)-Q(N_{f}(u)))+\varphi(S(u))\right])$.
\end{center}

Analogously to the section \ref{S:0}, here  $\varphi^{-1}$ is understood as the operator  $\varphi^{-1}: C \rightarrow C$ defined for $\varphi^{-1}(v)(t)=\varphi^{-1}(v(t))$. It is clear that $\varphi^{-1}$ is continuous and sends bounded sets into bounded sets.
\begin{lemma}\label{mate2}
$u \in C^{1}$ is  a solution  of (\ref{misto2}) if and only if  $u$ is a fixed point of the operator $M_{1}$.
\end {lemma}
\begin{proof}
Let $ u \in C^{1}$, we have the following equivalences:   
  
\medskip

$ (\varphi(u'))'  = N_f (u), \   u'(T)= u'(0), \ u'(0)=u(T)$
 \medskip

\medskip

\noindent
$ \Leftrightarrow(\varphi(u'))' = N_f(u)- Q(N_{f}(u))$,    

$Q(N_{f}(u))=0, \  u'(0)=u(T)$

\medskip

\medskip

\noindent
$ \Leftrightarrow \varphi(u')= H( N_f (u)-Q(N_{f}(u)))+\varphi(u'(0))$,
  
$Q(N_{f}(u))=0, \  u'(0)=u(T) $

 \medskip

\medskip

\noindent  
$ \Leftrightarrow u'= \varphi^{-1}\left[H( N_f (u)-Q(N_{f}(u)))+\varphi(u'(0))\right]$,
 
$Q(N_{f}(u))=0,  \  u'(0)=u(T)$ 

 \medskip

\medskip

\noindent 
$ \Leftrightarrow u= u(T)+ K(\varphi^{-1}\left[H( N_f (u)-Q(N_{f}(u)))+ \varphi(u'(0))\right])$,
 
$Q(N_{f}(u))=0, \  u'(0)=u(T) $

 \medskip

\medskip

\noindent    
$\Leftrightarrow  u= u(T)+ Q(N_{f}(u)) + K(\varphi^{-1}\left[H( N_f (u)-Q(N_{f}(u)))+\varphi(u(T))\right])$

 \medskip

\medskip

\noindent
$\Leftrightarrow  u= S(u)+ Q(N_{f}(u)) + K(\varphi^{-1}\left[H( N_f (u)-Q(N_{f}(u)))+\varphi(S(u))\right])$.
\end{proof}

\begin {remark}
 Note that  $u'(T)= u'(0) \Leftrightarrow  Q(N_{f}(u))=0$.
\end {remark}
Using an argument similar  to  the   introduced in Lemma \ref{parecido}, it is easy to see that, $M_{1}:C^{1}\rightarrow C^{1}$  is completely continuous.

In order to apply Leray-Schauder  degree to the operator $M_{1}$, we  introduced a family of problems  depending on a parameter $\lambda$. We remember  that to each  continuous function  $f:\left[0, T\right]\times \mathbb{R} \times \mathbb{R}\rightarrow \mathbb{R}$ we associate 
its Nemytskii operator $N_{f}: C^{1}\rightarrow C$ defined by
\begin{center}
$N_{f}(u)(t) = f(t, u(t), u'(t))$.
\end{center}
For  $ \lambda \in [0,1]$, we consider the family of boundary value problems
\begin{equation}\label{misto3}
\left\{\begin{array}{lll}
(\varphi(u' ))'   = \lambda N_{f}(u)+(1-\lambda)Q(N_{f}(u)) & & \\
u(T)=u'(0)=u'(T). 
\end{array}\right.
\end{equation}
 Notice that (\ref{misto3}) coincide with (\ref{misto2}) for $\lambda =1$. So, for  each $ \lambda \in [0,1]$, the operator associated  to \ref{misto3} by Lemma \ref{mate2} is the  operator $M(\lambda,\cdot)$, where $M$  is defined on $[0,1]\times C^{1}$  by 
 \begin{center}
$ M(\lambda,u)=S(u)+Q(N_{f}(u)) + K(\varphi^{-1}\left[\lambda H( N_f (u)-Q(N_{f}(u)))+\varphi(S(u))\right])$.
\end{center}
Using the same arguments as in the proof  of Lemma \ref{parecido}  we show that the operator  $M$ is completely continuous. Moreover, using the same reasoning as above, the system (\ref{misto3}) (see Lemma \ref{mate2}) is equivalent to the problem  
\begin{center}
$u=M(\lambda, u)$.
\end{center}
\subsection{Existence results}
\label{work2}
In this subsection, we present and prove our main results. These results are inspired on works by Bereanu and  Mawhin \cite{ma} and Man\'asevich and  Mawhin \cite{man}. We denote by $deg_{B}$ the Brouwer degree and  for $deg_{LS}$ the 
Leray-Schauder degree, and define the mapping $G:\mathbb{R}^{2}\rightarrow \mathbb{R}^{2}$ by 
\begin{center}
$G:\mathbb{R}^{2}\rightarrow \mathbb{R}^{2},  \   (a,b)\mapsto (aT+bT^{2}-bT-\frac{1}{T}\int_0^T f(t,a+bt,b)dt, b-a-bT)$.
\end{center} 
 \begin{theorem}\label{teoemaprincipalmisto}
Assume that  $\Omega$ is an open bounded set in  $C^{1}$ such that the following conditions hold.
\begin{enumerate}
\item For each $ \lambda \in (0,1)$ the problem
\begin{equation}\label{misto4}
\left\{\begin{array}{lll}
(\varphi(u' ))'   = \lambda N_{f}(u) & & \\
u(T)=u'(0)=u'(T), 
\end{array}\right.
\end{equation}
has no solution on $\partial\Omega$.
\item  The equation
\begin{center}
$G(a,b)=(0,0)$,
\end{center}
has no solution on $\partial\Omega \cap \mathbb{R}^{2}$, where we consider the natural identification $(a,b)\approx a+bt$ of  $\mathbb R^{2}$ with related functions in $C^{1}$.
\item  The Brouwer degree	
\begin{center}
$deg_{B}(G,\Omega \cap  \mathbb{R}^{2},0)\neq 0$.
\end{center}
\end{enumerate}
Then problem (\ref{misto2}) has a solution.
\end {theorem} 
\begin{proof}
Let  $ \lambda \in (0,1]$. If $u$ is  a solution of (\ref{misto4}), then $Q(N_{f}(u))=0$, hence  $u$ is a solution of problem (\ref{misto3}). On the other hand, for $ \lambda \in (0,1]$, if  $u$ is a solution  of (\ref{misto3}) and because 
\begin{center}
$ Q\left(\lambda N_{f}(u)+(1-\lambda)Q(N_{f}(u))\right)=Q(N_{f}(u))$,
\end{center}
we have  $Q(N_{f}(u))=0$, then  $u$ is a solution  of (\ref{misto4}). It follows that, for  $ \lambda \in (0,1]$, problems  (\ref{misto3}) and  (\ref{misto4}) have the same solutions. We assume that for  $\lambda=1$, (\ref{misto3}) does not have a solution on $\partial\Omega$  since otherwise we are done with proof. It follows that   (\ref{misto3}) has no solutions  for  $ (\lambda, u) \in  (0,1]\times \partial\Omega$. If  $\lambda=0$,  then (\ref{misto3}) is equivalent to the problem 
\begin{equation}\label{misto5}
\left\{\begin{array}{lll}
(\varphi(u' ))'   = Q(N_{f}(u)) & & \\
u(T)=u'(0)=u'(T), 
\end{array}\right.
\end{equation}
and thus, if $u$ is a solution  of  (\ref{misto5}), we must  have  
\begin{equation}\label{sapo}
\int_0^T f(t,u(t),u'(t))dt=0.
\end{equation}
Moreover, $u$  is a function of the form $u(t)=a + bt, \ a=b-bT $. Thus, by (\ref{sapo})
\begin{center}
$\int_0^T f(t,a+bt,b)dt=0$,
\end{center}
which, together with hypothesis 2, implies that $u= b-bT+tb \notin \partial\Omega$. Thus we have proved that  (\ref{misto3}) has no solution in  $\partial\Omega$ for all  $ \lambda \in \left[0, 1\right]$. Then we have that  for each  $ \lambda \in \left[0, 1\right]$, the  Leray-Schauder degree  $deg_{LS}(I-M(\lambda, \cdot), \Omega, 0)$ is well defined, and by the homotopy invariance  imply that  
\begin{center}
$deg_{LS}(I-M(0,\cdot),\Omega,0)= deg_{LS}(I-M(1,\cdot),\Omega,0)$.
\end{center}
On the other hand, we have that
\begin{center}
$deg_{LS}(I-M(0,\cdot),\Omega,0) = deg_{LS}(I-(S+QN_{f}+KS) ,\Omega,0)$.
\end{center}
But the range of the mapping
\begin{center}
$u \mapsto S(u)+Q(N_{f}(u))+K(S(u))$
\end{center}
is contained in the subspace  of related functions, isomorphic to $\mathbb{R}^{2}$. Thus, using a reduction property  of  Leray-Schauder degree \cite{man7, man12}
\begin{align*}
& deg_{LS}(I-(S+QN_{f}+KS) ,\Omega,0)\\
& = deg_{B}\left(I-(S+QN_{f}+KS)\left|_{\overline{\Omega \cap  \mathbb{R}^{2}}}\right. ,\Omega\cap \mathbb{R}^{2}, 0\right)\\
&=deg_{B}(G,\Omega\cap  \mathbb{R}^{2},0)\neq 0.
\end{align*}
Then, $deg_{LS}(I-M(1,\cdot),\Omega,0)\neq 0$, where $I$ denotes the unit operator. Hence, there exists  $u\in\Omega$ such that  $M_{1}(u)=u$, which is a 
solution  for (\ref{misto2}).
\end{proof}
 The following result gives a priori bounds for the possible solutions of   (\ref{misto4}), adapts a technique introduced by Ward \cite{wardjr}.
 \begin{theorem}\label{teopr2}
 Assume that  $f$ satisfies the following conditions.
\begin{enumerate}
\item There exists $c \in C$ such that 
\begin{center}
$f(t,x,y)\geq c(t)$
\end{center}
for all $(t,x,y)\in [0, T]\times \mathbb{R}\times \mathbb{R}$.
\item  There exists  $M_{1}<M_{2}$  such that for all  $u\in C^{1}$,
\begin{center}
$\int_0^T f(t,u(t),u'(t))dt \neq 0$ \ if \ $u_{m}'\geq M_{2} $,
\end{center}
\begin{center}
$\int_0^T f(t,u(t),u'(t))dt \neq 0$ \ if \ $u_{M}'\leq M_{1} $.
\end{center}
\end{enumerate}

If  $(\lambda,u)\in (0,1)\times C^{1} $ is such that  $u$ is  solution  of  (\ref{misto4}), then
\begin{center}
$\left\|u\right\|_{1}< r(2+T)$,
\end{center}
where $$r=\max\left\{\left|\varphi^{-1}(L+2\left\|c^{-}\right\|_{L^{1}})\right|, \left|\varphi^{-1}(-L-2\left\|c^{-}\right\|_{L^{1}})\right|\right\},  \  L=\max\left\{ \left| \varphi(M_{2})\right|, \left| \varphi(M_{1})\right| \right\}.$$
\end {theorem} 
\begin{proof}
Let   $(\lambda, u) \in (0,1)\times C^{1}$  be such that $u$ is a solution  of (\ref{misto4}). Then for all $t\in\left[0, T\right]$,
\begin{center}
$(\varphi(u'(t)))'  = \lambda N_f (u)(t),  \  u'(0)=u'(T)=u(T)$
\end{center}
and
\begin{center}
$ \int_ 0 ^T  f(t, u(t), u'(t)) dt=0$.
\end{center}
Using hypothesis 2, we have that
\begin{center}
$u_{m}'<M_{2}$  and  \  $u_{M}'> M_{1}$.
\end{center}
It follows that there exists  $\omega \in [0,T]$  such that  $M_{1}<u'(\omega)<M_{2}$  \  and  
\begin{center}
$\int_\omega ^t (\varphi(u'(s)))'ds= \lambda \int_\omega ^t  N_{f}(u)(s)ds$,
\end{center}
which implies that
\begin{center}
$\left| \varphi(u'(t))\right| \leq \left| \varphi(u'(\omega))\right|+  \int_0 ^T \left|f(s,u(s),u'(s))\right|ds$, 
\end{center}
where
\begin{center}
$\int_0 ^T \left|f(s,u(s),u'(s))\right|ds\leq \int_0 ^T f(s,u(s),u'(s))ds +2\int_0 ^T c^{-}(s)ds$.
\end{center}
Hence,
\begin{center}
$\left| \varphi(u'(t))\right| < L + 2\left\|c^{-}\right\|_{L^{1}} $,
\end{center}
where  $L= \max \left\{ \left| \varphi(M_{2})\right|, \left| \varphi(M_{1})\right| \right\}$ \  and  \ $t\in [0,T]$. It follows that  
\begin{center}
$\left\|u'\right\|_{\infty}< r$,
\end{center}
where  $r= \max\left\{\left|\varphi^{-1}(L+2\left\|c^{-}\right\|_{L^{1}})\right|,\ \left|\varphi^{-1}(-L-2\left\|c^{-}\right\|_{L^{1}})\right|\right\}$.
Because  $u\in C^{1}$ is such that  $u'(0)=u(T)$ we have that
\begin{center}
$\left|u(t)\right|\leq \left|u(T)\right|+ \int_ 0 ^T  |u'(s)| ds < r +rT   \  \  (t\in[0,T] ) $,
\end{center}
 and hence  $\left\|u\right\|_{1} =\left\|u\right\|_{\infty}+ \left\|u'\right\|_{\infty}< r+rT +r=r(2+T)$. This proves the theorem.
\end{proof} 
 Now we show the existence of at least one solution  for problem (\ref{misto2}) by means of Leray-Schauder degree. 
 \begin{theorem}\label{exem} 
Let $f$ be continuous and satisfy condition (1) and (2) of Theorem \ref{teopr2}. Assume that the following conditions hold for some   $\rho \geq r(2+T)$.  
\begin{enumerate}
\item  The equation
\begin{center}
$ G(a,b)=(0,0)$,
\end{center}
has no solution on  $\partial B_{\rho}(0)\cap\mathbb{R}^{2}$, where we consider the natural identification $(a,b)\approx a+bt$ of  $\mathbb R^{2}$ with related functions in $C^{1}$.
\item The Brouwer degree	
\begin{center}
$deg_{B}(G,B_{\rho}(0) \cap \mathbb{R}^{2},0)\neq 0$.
\end{center}
\end{enumerate}
Then problem (\ref{misto2}) has a solution.
\end{theorem} 
\begin{proof}
Let   $(\lambda, u) \in (0,1)\times C^{1}$  be such that $u$ is a solution  of (\ref{misto4}). Using  Theorem  \ref{teopr2}, we have 
\begin{center}
$\left\|u\right\|_{1} =\left\|u\right\|_{\infty}+ \left\|u'\right\|_{\infty}< r+rT +r=r(2+T)$,
\end{center}
where $r= \max\left\{\left|\varphi^{-1}(L+2\left\|c^{-}\right\|_{L^{1}})\right|,\ \left|\varphi^{-1}(-L-2\left\|c^{-}\right\|_{L^{1}})\right|\right\}$. Thus 
the conditions  of  Theorem \ref{teoemaprincipalmisto} are satisfied with  $\Omega=B_{\rho}(0)$,  where $B_{\rho}(0)$   is the open 
ball in  $C^{1}$  center $0$ and    radius $\rho$. This concludes the proof.
\end{proof}

Let us give now an application of Theorem \ref{exem}.
\begin{example}
Let us consider the problem 
\begin{equation}\label{exemplo}
\left((u')^{3}\right)' = \frac{e^{u'}}{2} -1, \quad u(T)=u'(0)=u'(T),
\end{equation}
for  $M_{1}=-1, \  M_{2}=1 ,  \  \rho\geq (1+2T)^{1/3}(2+T)$  and  $c(t)=-1$ for all $t\in \left[0, T\right]$. So, problem (\ref{exemplo}) has at least one solution. 
\end{example}


\section*{Acknowledgements} 
This research was supported by  CAPES and CNPq/Brazil. The author would like to thank to Dr. Pierluigi Benevieri for his kind advice and for the constructive revision of this paper.

\bibliographystyle{plain}

\renewcommand\bibname{References Bibliogr\'aficas}

\end{document}